%% file: Computable_quotient_presentations.tex
\documentclass[12pt]{amsart}

\title[]{Computable quotient presentations of models of arithmetic and set theory}

\author[M.~T.~Godziszewski]{Micha\l{} Tomasz Godziszewski}
 \address[M. T. Godziszewski]{Logic Department, Institute of Philosophy, University of Warsaw, Krakowskie Przedmiescie 3, 00-927 Warszawa}
 \email{mtgodziszewski@gmail.com}
 \urladdr{https://uw.academia.edu/MichalGodziszewski}

\author[Joel David Hamkins]{Joel David Hamkins}
 \address[J. D. Hamkins]
        {Mathematics, Philosophy, Computer Science,
          The Graduate Center of The City University of New York,
          365 Fifth Avenue, New York, NY 10016
          \&
          Mathematics,
          College of Staten Island of CUNY,
          Staten Island, NY 10314}
\email{jhamkins@gc.cuny.edu}
\urladdr{http://jdh.hamkins.org}
\thanks{This article is a preliminary report of results following up research initiated at the conference Mathematical Logic and its Applications, held in memory of Professor Yuzuru Kakuda of Kobe University in September 2016 at the Research Institute for Mathematical Sciences (RIMS) in Kyoto. The second author is grateful for the chance twenty years ago to be a part of Kakuda-sensei's logic group in Kobe, a deeply formative experience that he is pleased to see growing into a lifelong connection with Japan. He is grateful to the organizer Makoto Kikuchi and his other Japanese hosts for supporting this particular research visit, as well as to Bakhadyr Khoussainov for insightful conversations. The first author has been supported by the National Science Centre (Poland) research grant NCN PRELUDIUM UMO-2014/13/N/HS1/02058. He also thanks the Mathematics Program of the CUNY Graduate Center in New York for his research visit as a Fulbright Visiting Scholar between September 2016 and April 2017. Commentary concerning this paper can be made at \href{http://jdh.hamkins.org/computable-quotient-presentations}{http://jdh.hamkins.org/computable-quotient-presentations}.}

\usepackage{amssymb}
\usepackage[hidelinks]{hyperref}

\include{MathMacrosJDH}
\newcommand\cplus{\oplus}
\newcommand\ctimes{\odot}
\newcommand\clt{\vartriangleleft}
\newcommand\cleq{\trianglelefteq}
\newcommand\cin{\mathrel{\raise.1ex\hbox{$\epsilon$}}}
\newcommand\ZFfin{\text{\rm ZF}^{\neg\infty}}
\newcommand\Sing{\text{Sing}}
\newcommand\Doub{\text{Doub}}

\begin{document}

\begin{abstract}
We prove various extensions of the Tennenbaum phenomenon to the case of computable quotient presentations of models of arithmetic and set theory. Specifically, no nonstandard model of arithmetic has a computable quotient presentation by a c.e.~equivalence relation. No $\Sigma_1$-sound nonstandard model of arithmetic has a computable quotient presentation by a co-c.e.~equivalence relation. No nonstandard model of arithmetic in the language $\{+,\cdot,\leq\}$ has a computably enumerable quotient presentation by any equivalence relation of any complexity. No model of \ZFC\ or even much weaker set theories has a computable quotient presentation by any equivalence relation of any complexity. And similarly no nonstandard model of finite set theory has a computable quotient presentation.
\end{abstract}
\maketitle

\noindent A \emph{computable quotient presentation} of a mathematical structure $\mathcal A$ consists of a computable structure on the natural numbers $\<\N,\star,\ast,\dots>$, meaning that the operations and relations of the structure are computable, and an equivalence relation $E$ on $\N$, not necessarily computable but which is a congruence with respect to this structure, such that the quotient $\<\N,\star,\ast,\dots>/E$ is isomorphic to the given structure $\mathcal A$. Thus, one may consider computable quotient presentations of graphs, groups, orders, rings and so on, for any kind of mathematical structure. In a language with relations, it is also natural to relax the concept somewhat by considering the \emph{computably enumerable} quotient presentations, which allow the pre-quotient relations to be merely computably enumerable, rather than insisting that they must be computable.

At the 2016 conference Mathematical Logic and its Applications at the Research Institute for Mathematical Sciences (RIMS) in Kyoto, Bakhadyr Khoussainov \cite{Khoussainov2016:Computably-enumerable-structures-domain-independence-rims} outlined a sweeping vision for the use of computable quotient presentations as a fruitful alternative approach to the subject of computable model theory. In his talk, he outlined a program of guiding questions and results in this emerging area. Part of this program concerns the investigation, for a fixed equivalence relation $E$ or type of equivalence relation, which kind of computable quotient presentations are possible with respect to quotients modulo $E$.

In this article, we should like to engage specifically with two conjectures that Khoussainov had made in Kyoto. 

\newtheorem*{conjecture*}{Conjecture}
\begin{conjecture*}[Khoussainov]\label{Conjecture.Khoussainov}\
     \begin{enumerate}
       \item No nonstandard model of arithmetic admits a computable quotient presentation by a computably enumerable equivalence relation on the natural numbers.
       \item Some nonstandard model of arithmetic admits a computable quotient presentation by a co-c.e.~equivalence relation.
     \end{enumerate}
\end{conjecture*}

We shall prove the first conjecture and refute several natural variations of the second conjecture, although a further natural variation, perhaps the central case, remains open. In addition, we consider and settle the natural analogues of the conjectures for models of set theory.

Perhaps it will be helpful to mention as background the following observation, amounting to a version of the computable completeness theorem, which identifies a general method of producing computable quotient presentations.

\begin{observation}\label{Observation.Henkin-theory}
  Every consistent c.e.~theory $T$ in a functional language admits a computable quotient presentation by an equivalence relation $E$ of low Turing degree.
\end{observation}

\begin{proof}
Consider any computably enumerable theory $T$ in a functional language (no relation symbols). Let $\tau$ be the computable tree of attempts to build a complete consistent Henkin theory extending $T$, in the style of the usual computable completeness theorem. To form the tree $\tau$, we first give ourselves sufficient Henkin constants, and then add to $T$ all the Henkin assertions $\exists x\,\varphi(x)\implies\varphi(c_\varphi)$. Next, we enumerate all sentences in this expanded language, and then build the tree $\tau$ by adding to $T$ at successive nodes either the next sentence or its negation, provided that no contradiction has yet been realized from that theory by that stage. This tree is computable, infinite and at most binary branching. And so by the low basis theorem, it has a branch of low Turing complexity. Fix such a branch. The assertions made on it provide a complete consistent Henkin theory $T^+$ extending $T$. Let $A$ be the term algebra generated by the Henkin constants in the language of $T$. Thus, the elements of $A$ consist of formal terms in this language with the Henkin constants, and we may code the elements of $A$ with natural numbers. The natural operations on this term algebra are computable: to apply an operation to some terms is simply to produce another term. We may define an equivalence relation $E$ on $A$, by saying that two terms are equivalent $s\mathrel{E} t$, just in case the assertion $s=t$ is in the Henkin theory $T^+$, and this will be a congruence with respect to the operations in the term algebra, precisely because $T^+$ proves the equality axioms. Finally, the usual Henkin analysis shows that the quotient $A/E$ is a model of $T^+$, and in particular, it provides a computable quotient presentation of $T$.
\end{proof}

The previous observation is closely connected with a fundamental fact of universal algebra, namely, the fact that every algebraic structure is a quotient of the term algebra on a sufficient number of generators. Every countable group, for example, is a quotient of the free group on countably many generators, and more generally, every countable algebra (a structure in a language with no relations) arises as the quotient of the term algebra on a countable number of generators. Since the term algebra of a computable language is a computable structure, it follows that every countable algebra in a computable language admits a computable quotient presentation.

One of the guiding ideas of the theory of computable quotients is to take from this observation the perspective that the complexity of an algebraic structure is contained not in its atomic diagram, often studied in computable model theory, but rather solely in its equality relation. The algebraic structure on the term algebra, after all, is computable; what is difficult is knowing when two terms represent the same object. Thus, the program is to investigate which equivalence relations $E$ or classes of equivalence relations can give rise to a domain $\N/E$ for a given type of mathematical structure. There are many open questions and the theory is just emerging.

We should like to call particular attention to the fact that the proof method of observation \ref{Observation.Henkin-theory} and the related observation of univesal algebra breaks down when the language has relation symbols, because the corresponding relation for the resulting Henkin model will not generally be computable on the term algebra or even just on the constants. The complexity of the relation in the quotient structure arises from the particular branch that was chosen through the Henkin tree or equivalently from the Henkin theory itself. So it seems difficult to use the Henkin theory idea to produce computable quotient presentations of relational theories. We shall see later how this relational obstacle plays out in the case of arithmetic, whose usual language $\{+,\cdot,0,1,<\}$ includes a relation symbol, and especially in the case of set theory, whose language $\{\in\}$ is purely relational.

Let us now prove that Khoussainov's first conjecture is true.

\begin{theorem}\label{Theorem.c.e.}
 No nonstandard model of arithmetic has a computable quotient presentation by a c.e.~equivalence relation. Indeed, this is true even in the restricted (but fully expressive) language $\{+,\cdot\}$ with only addition and multiplication: there is no computable structure $\<\N,\cplus,\ctimes>$ and a c.e.~equivalence relation $E$, which is a congruence with respect to this structure, such that the quotient $\<\N,\cplus,\ctimes>/E$ is a nonstandard model of arithmetic.
\end{theorem}
%

\begin{proof}
Suppose toward contradiction that $E$ is a computably enumerable equivalence relation on the natural numbers, that $\<\N,\cplus,\ctimes>$ is a computable structure with computable binary operations $\cplus$ and $\ctimes$, that $E$ is a congruence with respect to these operations and that the quotient structure $\<\N,\cplus,\ctimes>/E$ is a nonstandard model of arithmetic. A very weak theory of arithmetic suffices for this argument.

Let $\bar 0$ be a number representing zero in $\<\N,\cplus,\ctimes>/E$ and let $\bar 1$ be a number representing one. Since $\cplus$ is computable, we can computably find numbers $\bar n$ representing the standard number $n$ in $\<\N,\cplus,\ctimes>/E$ simply by computing $\bar n=\bar 1\cplus\cdots\cplus\bar 1$.

Let $A$ and $B$ be computably inseparable c.e.~sets in the standard natural numbers. So they are disjoint c.e.~ sets for which there is no computable set containing $A$ and disjoint from $B$. Fix Turing machine programs $p_A$ and $p_B$ that enumerate $A$ and $B$, respectively. We shall run these programs inside the nonstandard model $\<\N,\cplus,\ctimes>/E$. Although every actual element of $A$ will be enumerated by $p_A$ inside the model at some standard stage, and similarly for $B$ and $p_B$, the programs $p_A$ will also enumerate nonstandard numbers into the sets, and it is conceivable that at nonstandard stages of computation, the program $p_A$ might place standard numbers into its set, even when those numbers are not in $A$. In particular, there is no guarantee in general that the sets enumerated by $p_A$ and $p_B$ in $\<\N,\cplus,\ctimes>/E$ will be disjoint.

Nevertheless, we proceed as follows. In the quotient structure, fix any nonstandard number $c$, and let $\tilde A$ be the set of elements below $c$ that in the quotient structure $\<\N,\cplus,\ctimes>/E$ are thought to be enumerated by $p_A$ before they are enumerated by $p_B$. Since every actual element of $A$ is enumerated by $p_A$ at a standard stage, and not by $p_B$ by that stage, it follows that the elements of $A$ are all in $\tilde A$, in the sense that whenever $n\in A$, then $\bar n$ is in $\tilde A$. Similarly, since the actual elements of $B$ are enumerated by $p_B$ at a standard stage and not by $p_A$ by that stage, it follows that none of the actual elements of $B$ will enter $\tilde A$.
\begin{equation*}
  \begin{split}
     n\in A &\quad\implies\quad \bar n\in \tilde A\\
     n\in B &\quad\implies\quad \bar n\notin\tilde A
  \end{split}
\end{equation*}
Thus, the set $C=\set{n\mid\bar n\in\tilde A}$ contains $A$ and is disjoint from $B$. We shall prove that $C$ is computable.

Since $\tilde A$ is definable inside $\<\N,\cplus,\ctimes>/E$, it is coded by an element of this structure. Let us use the prime-product coding method. Namely, inside the nonstandard model let $p_k$ be the $k^{th}$ prime number, and let $a$ be the product of the $p_k$ for which $k<c$ and $k\in \tilde A$.

Next, the key idea of the proof, we let $b$ be the corresponding code for the complement of $\tilde A$ below $c$. That is, $b$ is the product of the $p_k$ for which $k<c$ and $k\notin\tilde A$. We shall use both $a$ and $b$ to decode the set.

Given any number $n$, we can compute $\bar p_n$ and then search for a number $x$ for which $(x\ctimes \bar p_n)\mathrel{E} a$. In other words, we are searching for a witness that $\bar p_n$ divides $a$, from which we could conclude that $\bar n\in \tilde A$ and so $n\in C$. At the same time, we search for a number $y$ for which $(y\ctimes \bar p_n)\mathrel{E} b$. Such a $y$ would witness that $\bar p_n$ divides $b$ and therefore that $\bar n\notin\tilde A$ and hence $n\notin C$. The main point is that one or the other of these things will happen, since $a$ and $b$ code complementary sets, and so in this way we can compute whether $n\in C$ or not. So $C$ is a computable separation of $A$ and $B$, contrary to our assumption that they were computably inseparable.
\end{proof}

By replacing $x\ctimes\bar p_n$ in the proof with $x\cplus x\cplus\cdots\cplus x$, using $p_n$ many factors, we may deduce the Tennenbaum-style result that if $\<\N,\cplus,\ctimes>/E$ is a nonstandard model of arithmetic and $E$ is c.e., then $\cplus$ is not computable. That is, we don't need both operations in the pre-quotient structure to be computable. Similar remarks will apply to many of the other theorems in this article, and we shall explore this one-operation-at-a-time issue more fully in our follow-up article.

An alternative proof of theorem \ref{Theorem.c.e.} proceeds as follows. Consider the \emph{standard system} of any nonstandard model of arithmetic, which is the collection of traces on the standard $\N$ of the sets that are coded inside the model. Using the prime-product coding, for example, these can be seen as sets of the form $\set{n\mid \bar p_n\text{ divides }a}$, where $a$ is an arbitrary element of the model, $p_n$ means the $n^{th}$ prime number and $\bar p_n$ means the object inside the model that represents that prime number. It is a theorem of Scott that the standard systems of the countable nonstandard models of \PA\ are precisely the countable \emph{Scott} sets, which are sets of subsets of $\N$ that form a Boolean algebra, are closed downward under relative computability, and contain paths through any infinite binary tree coded in them. Because there is a computable tree with no computable path, every standard system must have noncomputable sets and therefore non-c.e.~sets, since it is closed under complements.

For the alternative proof of theorem \ref{Theorem.c.e.}, the main point is that the assumptions of the theorem ensure that every set in the standard system of the quotient model $\<\N,\cplus,\ctimes>/E$ is c.e., contradicting the fact we just mentioned. The reason is that for any object $a$, the number $n$ is in the set coded by $a$ just in case $\bar p_n$ divides $a$, and this occurs just in case there is a number $x$ for which $(x\ctimes\bar p_n)\mathrel{E} a$, which is a c.e.~property since $E$ is c.e.~and $\ctimes$ is computable. So every set in the standard system would be c.e., contrary to the fact we mentioned earlier.

Another alternative proof of a version of theorem \ref{Theorem.c.e.} handles the case of nonstandard models in the full language of arithmetic $\set{+,\cdot,0,1,<}$. Namely, if $E$ is c.e.~and $\<\N,\cplus,\ctimes,\bar 0,\bar 1,\clt>$ is a computably enumerable structure whose quotient by $E$ is a nonstandard model of arithmetic, then it follows from the next lemma that $E$ must also be co-c.e., and hence computable. And once we know that $E$ is computable, we may construct a computable nonstandard model of arithmetic, by using least representatives in each equivalence class, and this would contradict Tennenbaum's theorem, which says that there is no computable nonstandard model of arithmetic.

\begin{lemma}\label{Lemma.Automatic-c.e.-co-c.e.}
Suppose that $E$ is an equivalence relation on the natural numbers. 
\begin{enumerate}
  \item If $E$ is a congruence with respect to a computable relation $\clt$ and the quotient $\<\N,\clt>/E$ is a strict linear order, then $E$ is computable.
  \item If $E$ is a congruence with respect to a c.e. relation $\clt$ and the quotient $\<\N,\clt>/E$ is a strict linear order, then $E$ is co-c.e.
  \item If $E$ is a congruence with respect to a computable relation $\cleq$ and the quotient $\<\N,\cleq>/E$ is a reflexive linear order or  merely an anti-symmetric relation, then $E$ is computable.
  \item If $E$ is a congruence with respect to a c.e. relation $\cleq$ and the quotient $\<\N,\cleq>/E$ is a reflexive linear order $\cleq$ or merely anti-symmetric, then $E$ is c.e.
\end{enumerate}
\end{lemma}

\begin{proof}
For statement (1), suppose that $E$ is a congruence with respect to a computable relation $\clt$ and the quotient is a strict linear order. Since the quotient relation obeys 
    $$x\neq y\quad\iff\quad x<y\text{ or }y<x,$$
it follows that
    $$\neg(x\mathrel{E} y)\quad\iff\quad x\clt y\text{ or }y\clt x.$$
Since this latter property is computable, it follows that $E$ is computable. For statement (2), similarly, the latter property is c.e., and so $E$ is co-c.e.

For statement (3), suppose that $E$ is a congruence with respect to a computable relation $\cleq$, whose quotient is anti-symmetric. Since the quotient relation satisfies 
    $$x=y\quad \iff\quad x\leq y\text{ and }y\leq x,$$
it follows that 
    $$x\mathrel{E} y\quad\iff\quad x\cleq y\text{ and }y\cleq x.$$
If $\cleq$ is computable, as in statement (3), then $E$ will be computable. And if $\cleq$ is computably enumerable, as in statement (4), then $E$ must be c.e.
\end{proof}

In particular, including $<$ or $\leq$ in the language of arithmetic and asking for a computable or computably enumerable quotient presentation with respect to $E$ will impose certain complexity requirements on $E$, simply in order that $E$ is a congruence with respect to the order relation. 

Using this idea, the following corollary to theorem \ref{Theorem.c.e.} settles the version of Khoussainov's second conjecture for the language $\{+,\cdot,\leq\}$. By referring to the language of arithmetic with $\leq$, we intend the theory of arithmetic expressed in terms of the natural reflexive order relation, rather than the usual strict order relation $<$.

\begin{corollary}
No nonstandard model of arithmetic in the language $\{+,\cdot,\leq\}$ has a computably enumerable quotient presentation by any equivalence relation, of any complexity. That is, there is no computably enumerable structure $\<\N,\cplus,\ctimes,\cleq>$, where $\cplus$ and $\ctimes$ are computable binary operations and $\cleq$ is a computably enumerable relation, and an equivalence relation $E$ that is a congruence with respect to that structure, such that the quotient $\<\N,\cplus,\ctimes,\cleq>/E$ is a nonstandard model of arithmetic in the language $\{+,\cdot,\leq\}$.
\end{corollary}

\begin{proof}
Suppose toward contradiction that $E$ is an equivalence relation that is a congruence with respect to computable functions $\cplus$ and $\ctimes$ and c.e. relation $\cleq$ for which the quotient structure $\<\N,\cplus,\ctimes,\cleq>/E$ is a nonstandard model of arithmetic. Because the quotient of $\cleq$ by $E$ is a reflexive linear order, it follows by lemma \ref{Lemma.Automatic-c.e.-co-c.e.} that $E$ must be c.e., and so the corollary follows directly from theorem \ref{Theorem.c.e.}.
\end{proof}

Let's now consider another version of the second conjecture and the case of co-c.e.~equivalence relations. We shall refute the versions of the second conjecture for which the quotient model is to exhibit a certain degree of soundness.

Let's begin with an extreme version of this phenomenon, where we ask for far too much: models of true arithmetic. A model of \emph{true arithmetic} is a model with the same theory as the standard model of arithmetic. Equivalently, it is an elementary extension of the standard model inside it. After ruling out this extreme case, we shall than sharpen the result to the case of $\Sigma_1$-soundness and much less.

\begin{theorem}
 There is no computable structure $\<\N,\cplus,\ctimes>$ and a co-c.e.~equivalence relation $E$, which is a congruence with respect to this structure, such that the quotient $\<\N,\cplus,\ctimes>/E$ is a nonstandard model of true arithmetic.
\end{theorem}

\begin{proof}
Suppose that $\<\N,\cplus,\ctimes>$ is a computable structure and $E$ is a co-c.e.~equivalence relation, a congruence with respect to this structure, whose quotient $\<\N,\cplus,\ctimes>/E$ is a nonstandard model of true arithmetic. As in the earlier proof, let $\bar 1$ be a representative of the number $1$ inside this model and let $\bar n$ be the result of adding $\bar 1$ to itself $n$ times with $\cplus$ inside the model, so that $\bar n$ is a representative for what the quotient model thinks is the standard number $n$.

Since the quotient model satisfies true arithmetic, it follows that it is correct about the halting problem on standard numbers. So there is a number $h$ that codes the halting problem up to some nonstandard length $c$ of computations. In particular, for standard $n$ we shall have that $n\in 0'$ if and only if $\bar n$ is in the set coded by $h$. Another way to say this is that $0'$ is in the standard system of the quotient model, and this is all we actually require of true arithmetic here.

Let $A$ and $B$ be $0'$-computably inseparable sets, that is, sets that are computably enumerable relative to an oracle for the halting problem $0'$, but there is no $0'$-decidable separating set. Let $p_A$ and $p_B$ be the programs that enumerate $A$ and $B$ from an oracle for $0'$. Inside the nonstandard model $\<\N,\cplus,\ctimes>/E$, we may run $p_A$ and $p_B$ with the oracle determined by $h$, which happens to agree with $0'$ on the standard numbers. In particular, on standard input $n$, the computation with oracle $h$ inside the model will agree at the standard stages of computation with the actual computation using the real oracle $0'$.

Let $\tilde A$ be the elements $k<c$ that are enumerated by $p_A^h$ before they are enumerated by $p_B^h$. As before, our assumptions ensure that every actual element of $A$ is in $\tilde A$, and no element of $B$ is in $\tilde A$.
\begin{equation*}
  \begin{split}
     n\in A &\quad\implies\quad \bar n\in \tilde A\\
     n\in B &\quad\implies\quad \bar n\notin\tilde A
  \end{split}
\end{equation*}
Thus, the set $C$ of standard $n$ for which $\bar n\in\tilde A$ is a set that contains $A$ and is disjoint from $B$.

It remains for us to show for the contradiction that $C$ is computable from $0'$. As before, inside the quotient model, let $a$ be the product of $p_k$ for $k$ in $\tilde A$, and let $b$ be the product of $p_k$ for $k$ not in $\tilde A$. Given $n$, we want to determine whether $n\in C$ or not, which is equivalent to $\bar n\in\tilde A$. We can compute $\bar p_n$, and then we can try to discover if $\bar p_n$ divides $a$ or $\bar p_n$ divides $b$. Note that $\bar p_n$ divides $a$ just in case $\exists x\ (x\ctimes \bar p_n)\mathrel{E}a$, which has complexity $\Sigma_2$, since E is $\Pi_1$. Similarly, the relation $\bar p_n$ divides $b$ is also $\Sigma_2$. But since these answers are opposite, it follows that both of these relations are $\Delta_2$, and hence computable from $0'$. So the relation $n\in C$ is computable from $0'$, and we have therefore found a $0'$-computable separating set $C$, contradiction our assumption that $A$ and $B$ were $0'$-computably inseparable.
\end{proof}

We could alternatively have argued as in the alternative proof of theorem \ref{Theorem.c.e.} that every element of the standard system of the model is computable from $0'$, which is a contradiction if one knows that $0'$ is in the standard system.

Of course, true arithmetic was clearly much too strong in this theorem, and we could also have given a more direct alternative proof just by extracting higher-order arithmetic truths from this model in a $\Sigma_2$ or even $\Delta_2$-manner, since the pre-quotient model is computable and the relation is co-c.e. So a better theorem will eliminate or significantly weaken the true-arithmetic hypothesis, as we do in the following sharper result.

\begin{theorem}\label{Theorem.No-0'-by-co-c.e.}
There is no computable structure $\<\N,\cplus,\ctimes>$ and a co-c.e.~equivalence relation $E$, which is a congruence with respect to this structure, such that the quotient $\<\N,\cplus,\ctimes>/E$ is a $\Sigma_1$-sound nonstandard model of arithmetic, or even merely a nonstandard model of arithmetic with $0'$ in the standard system of the model.
\end{theorem}

\begin{proof}
 If the model is $\Sigma_1$-sound, then it computes the halting problem correctly, and so $0'$ will be in the standard system of the model, which means that it has a code $h$ as in the proof above. That was all that was required in the previous argument, and so the same contradiction is achieved.
\end{proof}

\begin{corollary}
No nonstandard model of arithmetic in the language $\set{+,\cdot,0,1,<}$ and with $0'$ in its standard system has a computably enumerable quotient presentation by any equivalence relation, of any complexity.
\end{corollary}

\begin{proof}
If $\<\N,\cplus,\ctimes,\bar0,\bar 1,\clt>/E$ is such a computably enumerable quotient presentation, then lemma \ref{Lemma.Automatic-c.e.-co-c.e.} shows that $E$ must be co-c.e., and so the situation is ruled out by theorem \ref{Theorem.No-0'-by-co-c.e.}.
\end{proof}

Note that containing $0'$ in the standard system is a strictly weaker property than being $\Sigma_1$-sound, since a simple compactness argument allows us to insert any particular set into the standard system of an elementary extension of any particular model of arithmetic.

Our results do not settle what might be considered the central case of the second conjecture, which remains open. We are inclined to expect a negative answer, whereas Khoussainov has conjectured a positive answer.

\begin{question}
 Is there a nonstandard model of \PA\ in the usual language of arithmetic $\singleton{+,\cdot,0,1,<}$ that has a computably enumerable quotient presentation by some co-c.e.~equivalence relation? Equivalently, is there a nonstandard model of \PA\ in that language with a computably enumerable quotient presentation by any equivalence relation, of any complexity?
\end{question}

The two versions of the question are equivalent by lemma \ref{Lemma.Automatic-c.e.-co-c.e.}, which shows that in the language with the strict order, the equivalence relation must in any case be co-c.e.

Let us now consider the analogous ideas for the models of set theory, rather than for the models of arithmetic. We take this next theorem to indicate how the program of computable quotient presentations has difficulties with purely relational structures.

\begin{theorem}\label{Theorem.ZFC-no-computable-quotients}
 No model of \ZFC\ has a computable quotient presentation. That is, there is no computable relation $\cin$ and equivalence relation $E$, a congruence with respect to $\cin$, for which the quotient $\<\N,\cin>/E$ is a model of \ZFC. Indeed, no such computable quotient is a model of \KP\ or even considerably weaker set theories.
\end{theorem}

Just to emphasize, we do not assume anything about the complexity of the equivalence relation $E$, which can be arbitrary, or about whether the quotient model of set theory $\<\N,\cin>/E$ is well-founded or ill-founded, standard or non-standard. Note also the typographic distinction between the relation $\cin$, which is the computable relation of the pre-quotient structure $\<\N,\cin>$, and the ordinary set membership relation $\in$ of set theory.

\begin{proof}
Suppose toward contradiction that $\cin$ is a computable relation on $\N$ and that $E$ is an equivalence relation, a congruence with respect to $\cin$, for which the quotient $\<\N,\cin>/E$ is a model of set theory. We need very little strength in the set theory, and even an extremely weak set theory suffices for the argument. We shall use the Kuratowski definition of ordered pair in set theory, for which $\<x,y>=\singleton{\singleton{x},\singleton{x,y}}$.

Since set theory proves that the set of natural numbers exists, there is some $N\in\N$ that the quotient model thinks represents the set of all natural numbers. Also, this model thinks that various kinds of sets involving natural numbers exist, such as the set coding the successor relation
 $$S=\set{\<n,n+1>\mid n\in\N}.$$
To be clear, we mean that $S$ is a number in $\N$ that the quotient model $\<\N,\cin>/E$ thinks is the set of the successor relation we identify above. So the $\cin$-elements of $S$ will all be thought to be Kuratowski pairs of natural numbers in the model, and this could include nonstandard numbers if there are any.

Similarly, we have sets consisting of the natural number singletons and doubletons.
\begin{equation*}
  \begin{split}
     \Sing&=\set{\singleton{n}\mid n\in\N},\\
     \Doub&=\set{\singleton{n,m}\mid n\neq m\text{ in }\N}.\\
  \end{split}
\end{equation*}
To be clear, we mean that $\Sing$ and $\Doub$ are particular elements of $\N$ that in the quotient model $\<\N,\cin>/E$ are thought to be the sets defined by those set-theoretic expressions. We assume that our set theory proves that these sets exist.

Next, I claim that there is a computable function $n\mapsto\bar n$, such that $\bar n$ represents what the quotient model $\<\N,\cin>/E$ thinks is the standard natural number $n$. To see this, we may fix a number $\bar 0$ that represents the number $0$. Next, given $\bar n$, we search for an element $d\cin S$ that will represent the pair $\<\bar n,m>$, and when found, we set $\overline{n+1}=m$. How shall we recognize this $d$ and $m$ using only $\cin$? Well, the $d$ we want has the form $\set{\singleton{\bar n},\ \singleton{\bar n,m}}$ inside the model, and so we search for an element $d\cin S$ that has an element $x\cin d$ with $x\cin\Sing$ and $\bar n\cin x$. This $x$ must represent the set $\singleton{\bar n}$, since $x$ is thought to have only one element, since it is in $\Sing$. Having found $d$, we search for $y\cin d$ with $y\in\Doub$ and an element $m$ with $m\cin y$, but $\neg (m\cin x)$. In this case, it must be that $y$ represents $\singleton{\bar n,m}$, and so we may let $\overline{n+1}=m$ and proceed. So the map $n\mapsto\bar n$ is computable.

It follows that every set in the standard system of the model $\<\N,\cin>/E$ is computable. Specifically, if $a$ is any element of the model, then the trace of this object on the natural numbers is the set $\set{n\in\N\mid \bar n\cin a}$, which would be a computable set, since both $\cin$ and the map $n\mapsto\bar n$ is computable.

But we mentioned earlier that every model of set theory and indeed of arithmetic must have non-computable sets in its standard system, so this is a contradiction.
\end{proof}

We could have argued a little differently in the proof. Namely, if $\cin$ is a computable relation with a congruence $E$ and $\<\N,\cin>/E$ is a model of set theory, then by the axiom of extensionality, we have
    $$x\neq y\quad\iff\quad\exists z\ \neg(z\in x\iff z\in y).$$
In the pre-quotient model, this amounts to:
    $$\neg(x\mathrel{E} y)\quad\iff\quad \exists z\neg(z\cin x\iff z\cin y).$$
Thus, in the case that $\cin$ is computable, in analogy with lemma \ref{Lemma.Automatic-c.e.-co-c.e.} we may deduce from this that $E$ must be co-c.e., even though we had originally made no assumption on the complexity of $E$. And in this case, the theorem follows from the next result.

Theorem \ref{Theorem.ZFC-no-computable-quotients} shows that it is too much to ask for computable quotient presentations of models of set theory. So let us relax the computability requirement on the pre-quotient membership relation $\cin$ by considering the case of computably enumerable quotient presentations, where $\cin$ is merely c.e.~rather than computable. In this case, we can still settle the second conjecture by ruling out quotient presentations by co-c.e.~equivalence relations.

\begin{theorem}\label{Theorem.ZFC-c.e.-quotients}
 There is no c.e.~relation $\cin$ with a co-c.e.~equivalence relation $E$ respecting it for which $\<\N,\cin>/E$ is a model of set theory.
\end{theorem}

\begin{proof}
In the proof of theorem \ref{Theorem.ZFC-no-computable-quotients}, we had used the computability of $\cin$, as opposed to the computable enumerability of $\cin$, in the step where we needed to know $\neg (m\cin x)$. At that step of the proof, really what we needed to know was that $m$ and $\bar n$ were not representing the same object. But if $E$ is co-c.e., then we can learn that $\bar n\neq m$ simply by waiting to see that $\bar n\mathrel{E} m$ fails, which if true will happen at some finite stage since $E$ is co-c.e. Indeed, it is precisely with the co-c.e.~equivalence relations $E$ that one is entitled to know by some finite stage that two numbers represent different objects in the quotient. Therefore, if $E$ is co-c.e., we still get a computable map $n\mapsto \bar n$. And then, in the latter part of the proof, we would conclude that every set in the standard system is c.e., since the trace of any object $a$ in the model on the natural numbers is the set $\set{n\in\N\mid \bar n\cin a}$, which would be c.e. But every standard system must contain non-c.e.~sets, by the paths-through-trees argument, since it contains non-computable sets and it is a closed under complements. So again we achieve a contradiction.
\end{proof}

Let us now explore the analogues of the earlier theorems for nonstandard models of finite set theory. Let $\ZFfin$ denote the usual theory of finite set theory, which includes all the usual axioms of ZFC, but without the axiom of infinity, plus the negation of the axiom of infinity and plus the $\in$-induction scheme formulation of the the foundation axiom. This theory is true in the structure $\<\HF,\in>$ of hereditarily finite sets, and it is bi-interpretable with \PA\ via the Ackermann relation on natural numbers.\footnote{Some researchers have also considered another strictly weaker version of this theory, omitting the $\in$-induction scheme. But it turns out that this version of the theory is flawed for various reasons: it cannot prove that every set has a transitive closure; it is not bi-interpretable with \PA; it does not support the Tennenbaum phenomenon (see \cite{EnayatSchmerlVisser2011:OmegaModelsOfFiniteSetTheory}). Meanwhile, since all these issues are addressed by the more attractive and fruitful theory $\ZFfin$, we prefer to take this theory as the meaning of `finite set theory.'}

\begin{theorem}\label{Theorem.ZFCfin-no-computable-quotients}
 There is no computable relation $\cin$ and equivalence relation $E$, a congruence with respect to $\cin$, of any complexity, such that the quotient $\<\N,\cin>/E$ is a nonstandard model of finite set theory $\ZFfin$.
\end{theorem}

\begin{proof}
Assume that $\cin$ is a computable relation for which $\<\N,\cin>/E$ is a nonstandard model of $\ZFfin$. The ordinals of this model with their usual arithmetic form a nonstandard model of \PA, which we may view as the natural numbers of the model. Let $N$ be a number representing a nonstandard such natural number in $\<\N,\cin>/E$. There is a set $S$ representing the set $\set{\<n,n+1>\mid n<N}$ as defined inside the model, and similarly we have sets representing the natural number singletons and doubletons up to $N$.
\begin{equation*}
  \begin{split}
     \Sing&=\set{\singleton{n}\mid n\in\N,\ n<N},\\
     \Doub&=\set{\singleton{n,m}\mid n\neq m\text{ in }\N,\ n,m<N}.\\
  \end{split}
\end{equation*}
So $\Sing$ and $\Doub$ are particular numbers in $\N$ that in the quotient $\<\N,\cin>/E$ represent the sets we have just defined by those expressions.

We may now run essentially the same argument as in the proof of theorem \ref{Theorem.ZFC-no-computable-quotients}. Namely, we may define a computable function $n\mapsto\bar n$, where $\bar n$ represents the natural number $n$ in the model $\<\N,\cin>/E$, by using the parameters $S$, $\Sing$ and $\Doub$ and decoding via the Kuratowski pair function as before. This argument uses the computability of $\cin$ as before in order to produce $\overline{n+1}$ from $\bar n$. Finally, we use this function to show that every set in the standard system of the model is computable, since for any object $a$, the trace of $a$ on the natural numbers is the set of $n$ for which $\bar n\cin a$, which is a computable property. This contradicts the fact that the standard system of any nonstandard model of $\ZFfin$ must include non-computable sets.
\end{proof}

Finally, we have the analogue of theorem \ref{Theorem.ZFC-c.e.-quotients} for the case of finite set theory.

\begin{theorem}
 There is no c.e.~relation $\cin$ with a co-c.e.~equivalence relation $E$ respecting it for which $\<\N,\cin>/E$ is a nonstandard model of finite set theory $\ZFfin$.
\end{theorem}

\begin{proof}
This theorem is related to theorem \ref{Theorem.ZFCfin-no-computable-quotients} the same way that theorem \ref{Theorem.ZFC-c.e.-quotients} is related to theorem \ref{Theorem.ZFC-no-computable-quotients}. Namely, in the proof of theorem \ref{Theorem.ZFCfin-no-computable-quotients}, we used the computability of the membership relation $\cin$ in the step computing the function $n\mapsto\bar n$. If $\cin$ is merely computably enumerable, as here, then we can nevertheless still find a computable function $n\mapsto\bar n$, provided that the equivalence relation $E$ is co-c.e., since in the details of the proof as explained in theorem \ref{Theorem.ZFC-c.e.-quotients}, we needed to know that we had found the right value for $\overline{n+1}$ by knowing that a certain number $m$ was actually representing a different number than $\bar n$, and it is precisely with a co-c.e.~equivalence relation $E$ that one can know such a thing at some finite stage.

If $n\mapsto\bar n$ is computable, then with a c.e.~relation $\cin$, we can deduce that every set in the standard system is c.e., since $a$ codes the set of $n$ for which $\bar n\cin a$, a c.e.~property, and this contradicts the fact that every standard system of a nonstandard model of $\ZFfin$ must contain non-c.e.~sets.
\end{proof}

We expect to follow up this article with a second article containing several more refined results.

\bibliographystyle{alpha}
\bibliography{MathBiblio,HamkinsBiblio,WebPosts}

\end{document}

%% file: MathMacrosJDH.tex
\newtheorem{theorem}{Theorem}

\newtheorem*{maintheorem*}{Main Theorem}
\newtheorem*{maintheorems*}{Main Theorems}
\newtheorem{corollary}[theorem]{Corollary}
\newtheorem*{corollary*}{Corollary}
\newtheorem*{corollaries*}{Corollaries}

\newtheorem{lemma}[theorem]{Lemma}

\newtheorem{question}[theorem]{Question}

\newtheorem*{questions*}{Questions}
\newtheorem*{mainquestion*}{Main Question} 
\newtheorem*{openquestion*}{Open Question} 
\newtheorem{observation}[theorem]{Observation}

\newcommand{\QED}{\end{proof}}

\def\proclaim[#1]{{\bf #1}}
\def\BF#1.{{\bf #1.}}

\newcommand{\N}{{\mathbb N}}

\newcommand{\set}[1]{\{\,{#1}\,\}}
\newcommand{\singleton}[1]{\left\{{#1}\right\}}

\newcommand{\smalllt}{\mathrel{\mathchoice{\raise2pt\hbox{$\scriptstyle<$}}{\raise1pt\hbox{$\scriptstyle<$}}{\raise0pt\hbox{$\scriptscriptstyle<$}}{\scriptscriptstyle<}}}
\newcommand{\smallleq}{\mathrel{\mathchoice{\raise2pt\hbox{$\scriptstyle\leq$}}{\raise1pt\hbox{$\scriptstyle\leq$}}{\raise1pt\hbox{$\scriptscriptstyle\leq$}}{\scriptscriptstyle\leq}}}

\newcommand{\boolval}[1]{\mathopen{\lbrack\!\lbrack}\,#1\,\mathclose{\rbrack\!\rbrack}}
\def\[#1]{\boolval{#1}}
\newbox\gnBoxA
\newdimen\gnCornerHgt
\setbox\gnBoxA=\hbox{\tiny$\ulcorner$}
\global\gnCornerHgt=\ht\gnBoxA
\newdimen\gnArgHgt
\def\gcode #1{%
\setbox\gnBoxA=\hbox{$#1$}%
\gnArgHgt=\ht\gnBoxA%
\ifnum     \gnArgHgt<\gnCornerHgt \gnArgHgt=0pt%
\else \advance \gnArgHgt by -\gnCornerHgt%
\fi \raise\gnArgHgt\hbox{\tiny$\ulcorner$} \box\gnBoxA %
\raise\gnArgHgt\hbox{\tiny$\urcorner$}}
\newcommand{\UnderTilde}[1]{{\setbox1=\hbox{$#1$}\baselineskip=0pt\vtop{\hbox{$#1$}\hbox to\wd1{\hfil$\sim$\hfil}}}{}}
\newcommand{\Undertilde}[1]{{\setbox1=\hbox{$#1$}\baselineskip=0pt\vtop{\hbox{$#1$}\hbox to\wd1{\hfil$\scriptstyle\sim$\hfil}}}{}}
\newcommand{\undertilde}[1]{{\setbox1=\hbox{$#1$}\baselineskip=0pt\vtop{\hbox{$#1$}\hbox to\wd1{\hfil$\scriptscriptstyle\sim$\hfil}}}{}}
\newcommand{\UnderdTilde}[1]{{\setbox1=\hbox{$#1$}\baselineskip=0pt\vtop{\hbox{$#1$}\hbox to\wd1{\hfil$\approx$\hfil}}}{}}
\newcommand{\Underdtilde}[1]{{\setbox1=\hbox{$#1$}\baselineskip=0pt\vtop{\hbox{$#1$}\hbox to\wd1{\hfil\scriptsize$\approx$\hfil}}}{}}

\renewcommand{\implies}{\mathrel{\rightarrow}}

\renewcommand{\iff}{\mathrel{\leftrightarrow}}

\def\<#1>{\left\langle#1\right\rangle}

\newcommand{\ZFC}{{\rm ZFC}}

\newcommand{\KP}{{\rm KP}}

\newcommand{\HF}{{\rm HF}}

\newcommand{\PA}{{\rm PA}}

\newcommand{\cell}[1]{\boxit{\hbox to 17pt{\strut\hfil$#1$\hfil}}}
\newcommand{\head}[2]{\lower2pt\vbox{\hbox{\strut\footnotesize\it\hskip3pt#2}\boxit{\cell#1}}}
\newcommand{\boxit}[1]{\setbox4=\hbox{\kern2pt#1\kern2pt}\hbox{\vrule\vbox{\hrule\kern2pt\box4\kern2pt\hrule}\vrule}}
\newcommand{\Col}[3]{\hbox{\vbox{\baselineskip=0pt\parskip=0pt\cell#1\cell#2\cell#3}}}
\newcommand{\tapenames}{\raise 5pt\vbox to .7in{\hbox to .8in{\it\hfill input: \strut}\vfill\hbox to
.8in{\it\hfill scratch: \strut}\vfill\hbox to .8in{\it\hfill output: \strut}}}
\newcommand{\Head}[4]{\lower2pt\vbox{\hbox to25pt{\strut\footnotesize\it\hfill#4\hfill}\boxit{\Col#1#2#3}}}
\newcommand{\Dots}{\raise 5pt\vbox to .7in{\hbox{\ $\cdots$\strut}\vfill\hbox{\ $\cdots$\strut}\vfill\hbox{\
$\cdots$\strut}}}

%% file: Computable_quotient_presentations.bbl
\begin{thebibliography}{ESV11}

\bibitem[ESV11]{EnayatSchmerlVisser2011:OmegaModelsOfFiniteSetTheory}
A.~Enayat, J.~Schmerl, and A.~Visser.
\newblock $\omega$-models of finite set theory.
\newblock In J.~Kennedy and R.~Kossak, editors, {\em Set theory, Arithmetic,
  and Foundations of Mathematics: Theorems, Philosophies}, number~36 in Lecture
  Notes in Logic, chapter~4. Cambridge University Press, 2011.

\bibitem[Kho16]{Khoussainov2016:Computably-enumerable-structures-domain-independence-rims}
Bakhadyr Khoussainov.
\newblock Computably enumerable structures: Domain dependence, September 2016.
\newblock slides for conference talk at Mathematical Logic and its
  Applications, Research Institute for Mathematical Sciences (RIMS), Kyoto
  University, \url{http://www2.kobe-u.ac.jp/~mkikuchi/mla2016khoussainov.pdf}.

\end{thebibliography}
